\documentclass[a4paper,11pt]{amsart}%{article}%

\usepackage{amsmath, amsthm, amssymb, graphicx, amsrefs}
\usepackage{color}

\usepackage[top=3.0cm, bottom=3.0cm, inner=3.0cm, outer=3.0cm, includefoot]{geometry}
\usepackage[latin1]{inputenc}
\usepackage[T1]{fontenc}
\usepackage{verbatim}
\usepackage{geometry}
\usepackage{amssymb}
\usepackage{amsmath}
\usepackage{graphicx}
\usepackage{amsthm}
\usepackage{color}
\usepackage[pdftex]{hyperref}

\newcommand{\Hm}[1]{\leavevmode{\marginpar{\tiny%
$\hbox to 0mm{\hspace*{-0.5mm}$\leftarrow$\hss}%
\vcenter{\vrule depth 0.1mm height 0.1mm width \the\marginparwidth}%
\hbox to 0mm{\hss$\rightarrow$\hspace*{-0.5mm}}$\\\relax\raggedright
#1}}}
\newcommand{\eat}[1]{}

\newcommand{\IR}{{\mathbb{R}}}
\newcommand{\R}{{\mathbb{R}}}

\newcommand{\N}{{\mathbb{N}}}

\newcommand{\eps}{{\varepsilon}}

\newcommand{\Gm}{{\Gamma}}

\newcommand{\eChar}{\begin{enumerate}[(i)]}
\newcommand{\eBr}{\begin{enumerate}[(1)]}

\newcommand{\supp}{\operatorname{supp}}

\newcommand{\ind}{k}

\theoremstyle{plain}
\newtheorem{lemma}{Lemma}[section]
\newtheorem{theorem}[lemma]{Theorem}
\newtheorem{proposition}[lemma]{Proposition}
\newtheorem{corollary}[lemma]{Corollary}

\theoremstyle{definition}

\newtheorem{example}[lemma]{Example}

\newtheorem{rem}[lemma]{Remark}

\newcommand{\Deg}{\operatorname{Deg}}

\numberwithin{equation}{section}

\title{Gradient estimates, Bakry-Emery Ricci curvature and  ellipticity for unbounded graph Laplacians}
\author{Matthias Keller, Florentin M\"unch}
\date{\today}

\begin{document}

\maketitle

\begin{abstract}
In this article, we prove  gradient estimates  under Bakry-Emery curvature bounds for unbounded graph Laplacians which satisfy an ellipticity assumption. As applications, we study completeness and finiteness of stochastically complete graphs under Bakry-Emery curvature bounds.
\end{abstract}
%\tableofcontents

\pagestyle{plain}

\section{Introduction}
It has been known for many decades that there is a deep relationship between Ricci curvature and the heat equation.
In particular, lower Ricci curvature bounds can be characterized via gradient estimates for the heat semigroup \cites{renesse2005transport, wang2011equivalent}.
Recently, there has been remarkable interest in discrete versions of Ricci curvature. Specifically, numerous definitions and approaches have been put forward. For example, 
Ricci curvature via optimal transport has been studied in
\cites{bauer2011ollivier,bauer2015generalized,bhattacharya2015exact,jost2014ollivier,joulin2009new,lin2011ricci,ollivier2007ricci,ollivier2009ricci,ollivier2010survey,munch2017ollivier,bourne2017ollivier,sandhu2015graph,sandhu2016ricci,tannenbaum2015ricci,paeng2012volume,loisel2014ricci,ni2015ricci,veysseire2012coarse,wang2014wireless,wang2016interference,BJL12,eldan2017transport,lin2013ricci,cushing2018ricci,cushing2018erratum,ollivier2012curved,rubleva2016ricci,cushing2018rigidity,kamtue2018combinatorial,cushing2018curvature}. Ricci curvature defined by the convexity of entropy has been studied in \cites{erbar2012ricci,erbar2018poincare,fathi2016entropic,MAAS20112250}.
Discrete Bakry Emery Ricci curvature based on Bochner's formula builds on \cites{bakry1985diffusions} and has been studied in \cites{chung2017curvature,chung2014harnack,cushing2016bakry,Gong2017,horn2014volume,hua2017stochastic,fathi2018curvature,kamtue2018combinatorial,klartag2015discrete,lin2015equivalent,liu2014eigenvalue,liu2015curvature,liu2016bakry,munch2014li,bauer2015li,lin2010ricci,munch2017remarks,hua2017stochastic,schmuckenschlager1998curvature,kempton2017relationships,liu2017rigidity,hua2017ricci,gao2016curvature,gao2016one,HL16,johnson2015discrete,liu2017distance,yamada2017curvature,cushing2018curvature}.
 Specifically, via Bakry Emery curvature, analogous gradient estimates as on manifolds have been established on graphs  under several additional assumptions \cites{hua2017stochastic, lin2015equivalent, Gong2017}.
However, most of this work is centered around graphs with bounded vertex degree. Only recently unbounded Laplacians were studied \cites{GLLY18,hua2017stochastic,Gong2017}. In this article, we prove a characterization of various gradient estimates for unbounded Laplacians by means of  Bakry-Emery Ricci curvature under an ellipticity assumption.

We say a locally finite graph $ b $ over a discrete measure space $ (X,m) $ (see Section~\ref{s:setup} for a definition) satisfies the ellipticity condition $ \mathrm{(EC)} $ if there is a constant $ C>0 $ such that
\begin{align*}\tag{$ \mathrm{EC} $}
b(x,y)\leq Cm(x)m(y),\qquad x,y\in X.
\end{align*}
%Furthermore, we denote the space of  all  positive, finitely supported functions by $ C_{c}^{+}(X)$.

The main theorem of this paper provides  a characterization of the Bakry-Emery curvature condition 
$CD(K,n)$  via various gradient estimates of the Laplacian $ \Delta $ and its semigroup $ P_{t} $ in terms of the carr\'e du champ operator $ \Gamma $ on the space of positive finitely supported functions $ C_{c}^{+}(X) $ (see Section~\ref{s:setup} for definitions).

\begin{theorem}\label{thm:main}
	Let $ b $ be a locally finite graph over $ (X,m) $ that satisfies the ellipticity condition $ \mathrm{(EC)} $.
	Let $ n>0 $ and $ K\in \R$. Then, the following statements are equivalent:
	\begin{align*}
	\mathrm{(i)}&&& CD(-K,n) \mbox{ holds},\\
\mathrm{(ii)}&&&	\Gamma(P_{t}f) \leq e^{2Kt}P_{t}(\Gamma(f))-\frac{2}{n}\int_{0}^{t}e^{2Ks}P_{s}\left(P_{t-s} \Delta f\right)^{2} ds, && f\in C_{c}^{+}(X),\, t\ge0,\\
\mathrm{(iii)}&&&\Gamma(P_{t}f) \leq e^{2Kt}P_{t}(\Gamma(f))-\frac{e^{2Kt}-1}{Kn}\left(P_{t} \Delta f\right)^{2}, && f\in C_{c}^{+}(X),\, t\ge0,\\
\mathrm{(iv)}&&&P_{t}f^{2}-(P_{t}f)^{2}\leq \frac{e^{2Kt}-1}{K}P_{t}(\Gamma(f))-\frac{e^{2Kt}-1-2Kt}{K^{2}n}\left(P_{t} \Delta f\right)^{2}, && f\in C_{c}^{+}(X),\, t\ge0,\\
\mathrm{(v)}&&&P_{t}f^{2}-(P_{t}f)^{2}\geq \frac{1-e^{-2Kt}}{K}\Gamma(P_{t}f)+\frac{e^{-2Kt}-1+2Kt}{K^{2}n}\left(P_{t} \Delta f\right)^{2}, && f\in C_{c}^{+}(X),\, t\ge0,
	\end{align*}
	where for $ K=0 $ we set $ (1-e^{\pm 2Kt})/K=\mp 2t $ and   $ (e^{\pm 2Kt}-1\mp 2Kt)/K^2=2t ^{2} $.
\end{theorem}
Let us discuss the result in view of the literature.

For bounded graph Laplacians, such a theorem was proven in
\cite[Theorem~3.1]{lin2015equivalent}.
Moreover, in \cite{hua2017stochastic} the equivalence of (i) $ \Longleftrightarrow $ (ii) for $ n=\infty $ was proven under the assumption of completeness of the graph (see Section~\ref{s:appl}) and non-degeneracy of the measure, i.e., $ m\ge C $ for some constant $ C>0 $. Later in \cite{Gong2017},    (i) $ \Longleftrightarrow $ (iv)  was proven for $ n=\infty $  and  (i) $ \Longleftrightarrow $ (ii) was proven for arbitrary positive $ n$  under the same assumptions, i.e., completeness and non-degenerate measure.

From a logical point of view, our condition (EC) is independent from the conditions in the literature. Specifically, (EC) is independent  of boundedness as well as of non-degeneracy of the measure and completeness. Hence, our result stands  skew to the results discussed above. %However, we stress that (EC) clearly allow for a large class of unbounded graphs.

  From a conceptual  point of view, our assumption (EC) is  purely local and geometrical. This stands in contrast to the non-local completeness  assumption which is rather analytical and has no immediate geometric meaning.

  Finally, from a practical point of view, (EC) is explicit and allows for unbounded graphs.  On the other hand,  completeness is often hard to check and boundedness is rather restrictive.

  A major advantage of our approach is that graphs with standard weights and the counting measure as well as with the normalizing measure satisfy (EC).

The paper is structured as follows. In the next section we introduce the basic notions. In Section~\ref{s:green} we prove a Green's formula which is a consequence of the ellipticity condition (EC). In Section~\ref{s:product} we compute the derivative of the function $ s\mapsto P_{s}\Gamma(P_{t-s}f)(x) $ with respect to $ s
 $
under very mild conditions. We put these two pieces together in Section~\ref{s:proof} to prove the theorem above. We next study applications of the main theorem for stochastically complete graphs.
Finally, we give examples of graphs for which our method gives 
gradient estimates, but all previous approaches fail to do so.

\section{Setup and notations}\label{s:setup}

Let a  discrete countable space $ X $ be given.
A strictly positive function $ m $ on $ X $ extends to a measure of full support via additivity, i.e., $ m(A)=\sum_{x\in A}m(x) $, $ A\subseteq X $, and  $ (X,m) $ is referred to as a \emph{discrete measure space}. 
A \emph{locally finite graph}  over a discrete measure space $ (X,m) $ is a symmetric function $ b:X\times X\to[0,\infty) $ with zero diagonal such that the sets $\{y\in X\mid b(x,y)>0\}  $ are finite for all $ x\in X $. The graph  $ b $ is called \emph{connected} if for every $x,y\in X$ there are $x_1,\ldots, x_n\in X$ such that ${b(x,x_1), b(x_1,x_2),\ldots, b(x_n,y)>0}$.

The space of real valued functions on $ X $ is denoted by $ C(X) $ and its subspace of compactly, and, thus, finitely supported functions is denoted by $ C_{c}(X) $.
We denote by $ \ell^{p}(X,m) $  the standard spaces of $ p $-summable real valued functions with respect to $ m $ and $ \ell^{\infty}(X) $ the space of bounded real valued functions.  Their norms are denoted by $ \|\cdot\|_{p} $, $ p\in[1,\infty] $. For $ p,q\in[1,\infty] $ such that $ 1/p+1/q=1 $, we denote the dual pairing by $ \langle\cdot,\cdot\rangle_{p,q} $. Furthermore, for $ f,g\in C(X) $, we write
\begin{align*}
\langle f,g\rangle_{\mathrm{abs}} :=\sum_{x\in X}f(x) g(x)m(x),
\end{align*}
whenever the right hand side converges absolutely.

On $ C(X) $, we define the \emph{Laplacian} $ \Delta $ as
\begin{align*}
\Delta f(x)=\frac{1}{m(x)}\sum_{y\in X}b(x,y)(f(y)-f(x)).
\end{align*}
With the Laplacian at hand, we can introduce the $ \Gamma $-calculus (see e.g. \cite{bakry1985diffusions}) via setting 
$ \Gm_{0}(f,g)=f\cdot g $ and iteratively
\[
2\Gamma_{\ind+1} (f,g) := \Delta\Gamma_\ind (f,g) - \Gamma_\ind(f,\Delta g) -  \Gamma_\ind(\Delta f,g)
\]
for $ f,g\in C(X) $ and $ k\in\N_0 $.
For convenience, we write $\Gamma_k (f) := \Gamma_k(f,f)$ and $\Gamma := \Gamma_1$ which is often referred to as the \emph{carr\'e du champ operator}. Obviously, $ \Gm_{k} $ is bilinear.

We say a graph satisfies the \emph{curvature dimension condition} $CD(K,n)$ with curvature $K \in \IR$ and dimension $n >0$ if, for all $f \in C(X)$,
\[
\Gamma_2(f) \geq \frac 1 n (\Delta f)^2 + K \Gamma(f).
\]

In order to introduce the heat semigroup, we need  to consider a self-adjoint restriction of $ \Delta $. 
Due to the local finiteness of the graph, the restriction of $ -\Delta $ to $ C_{c}(X) $ is  a symmetric positive operator on $\ell^{2}(X,m)  $. Thus,  $ -\Delta\vert_{ C_{c}(X)}  $ has a Friedrichs extension $ L $ with domain $ D(L) $, see e.g. \cite{weidmann2012linear}. 
By the virtue of a Green formula, it is not hard to see (confer \cite{keller2012dirichlet}) that the domain of the Friedrichs extension is given by
\[
D(L)=\{f \in D(Q) : - \langle \Delta f,g \rangle_{2,2} = Q(f,g) \mbox{ for all } g \in D(Q) 
 \}
\]
with
$Q(f,g)=\langle \Gamma(f,g),1\rangle_{1,\infty}$
for all $f,g \in C(X)$ with $\Gamma(f)\in \ell_1(X,m)$ and $\Gamma(g) \in \ell_1(X,m)$, and
$
D(Q) = \overline{C_c(X)}^{\|\cdot\|_Q}
$
with
\[
\|f\|_Q = Q(f,f) + \|f\|_2^2.
\]

One can observe easily that $ L $ is a restriction of $ -\Delta $, i.e.,
\begin{align*}
L=-\Delta\quad \mbox{on }D(L).
\end{align*}
By the spectral theorem we define powers  $ L^{k} $, $ k\ge0 $, of $ L $. By local finiteness of the graph, the compactly supported functions $ C_{c}(X) $ are included in $ D(L^{k}) $ for all $ k\ge0 $. Furthermore, we define the heat semigroup
\begin{align*}
P_{t}:=e^{-tL},\qquad t\ge0,
\end{align*}
which extends to a bounded positive contraction semigroup on all $ \ell^{p}(X,m) $ spaces, $ p\in [1,\infty] $ and is strongly continuous for $ p\in[1,\infty) $. Moreover, for  $ f\in \ell^{2}(X,m) $, the function
\begin{align*}
u_{t}:=P_{t}f
\end{align*}
is the unique solution of the heat equation
\begin{align*}
\Delta u_{t}=\partial_{t}u_{t},\qquad u_{0}=f.
\end{align*}
 in $ D(L) $.
Corresponding statements hold true for $ f \in\ell^{p}(X,m)$, $ p\in[1,\infty) $, with uniqueness in the domain of the generator of $ P_{t} $ in $ \ell^{p} (X,m)$.
Moreover, the heat semigroup on non-negative functions in $\ell^\infty(X)$ can also be introduced as the smallest non-negative solution to the heat equation, see \cite{keller2012dirichlet, wojciechowski2008heat}.

\section{Green's formula}\label{s:green}

It can be seen that the quadratic form of the restriction of $ -\Delta $ is given by
\begin{align*}
\|\Gamma(f)\|_{1}=\| L^{1/2}f\|_{2}^{2}, \qquad f\in D(L^{1/2}),
\end{align*}
and by definition of $ L $ and $ -\Delta=L $ we have
\begin{align*}
\langle{ f, \Delta g}\rangle_{2,2}= -\sum_{x\in X}\Gamma(f,g)(x)m(x)=\langle{\Delta f,g}\rangle_{2,2}\qquad f,g\in D(L).
\end{align*}
Below, we will extend this equality known as \emph{Green's formula} to a larger class of functions under the condition (EC).% which asserts the existence of $ C\ge0  $ such that $ b(x,y)\leq C m(x)m(y) $ for all $ x,y\in X $.
\begin{proposition}[Green's formula]\label{p:GF} Assume the  graph $ b $ over $(X,m)  $ satisfies (EC). Then, for all  $ f,g\in \ell^{1}(X,m) $ with $ \Delta g\in \ell^{\infty}(X) $ we have
	\begin{align*}
	\langle f,\Delta g\rangle_{1,\infty} =	\langle \Delta f, g\rangle_{\mathrm{abs}}.
	\end{align*}	
\end{proposition}

To prove the statement, consider the weighted vertex degree $ \Deg:X\to[0,\infty) $
\begin{align*}
\Deg(x)=\frac{1}{m(x)}\sum_{y\in X}b(x,y),\qquad x\in X,
\end{align*}
and we denote the corresponding multiplication operator on $ C(X) $ by slight abuse of notation also by $ \Deg $. This way the \emph{adjacency matrix} $ A:=\Delta+\Deg $ acts on $ C(X) $ as
\begin{align*}
Af(x)=\frac{1}{m(x)}\sum_{y\in X}b(x,y)f(y),\qquad x\in X.
\end{align*}
We can characterize the ellipticity condition (EC) by boundedness of the adjacency matrix as an operator from $ \ell^{1}(X,m) $ to $ \ell^{\infty}(X) $.

\begin{lemma}[Characterizing (EC) via $ A $] Let $ b $ be a graph over $ (X,m) $. The following statements are equivalent:
	\begin{itemize}
		\item[(i)] $ \mathrm{(EC)} $ holds, i.e. there is $ C\ge0  $ such that $ b(x,y)\leq C m(x)m(y) $ for all $ x,y\in X $.
	\item[(ii)]  The operator $ A $ is  bounded  from $ \ell^{1}(X,m) $ to $ \ell^{\infty}(X) $.
	\end{itemize}
\end{lemma}
\begin{proof}
	(i) $ \Longrightarrow $ (ii): The boundedness of $ A $ from $ \ell^{1}(X,m) $ to $ \ell^{\infty}(X) $ follows from the estimate
	\begin{align*}
	|Af(x)| \leq \frac 1 {m(x)} \sum_{y\in X} b(x,y) |f(y)| \leq C \sum_{y\in X} m(y)|f(y)| = C\|f\|_1,\qquad x\in X.
	\end{align*}
	Hence, $ \|A\|_{1 \to \infty}\leq C $ with $ \|\cdot\|_{1 \to \infty} $ being the operator norm from $ \ell^{1}(X,m) $ to $ \ell^{\infty}(X) $.
	
	(ii) $ \Longrightarrow $ (i):	The boundedness of $ A $ from $ \ell^{1}(X,m) $ to $ \ell^{\infty}(X) $ applied to the characteristic function $ 1_{y} $ of the vertex $ y $ gives
		\begin{align*}
		\frac{b(x,y)}{m(x)}=A1_{y}(x) \leq \|A 1_{y}\|_\infty \leq \|A\|_{1 \to \infty}\|1_{y}\|_1 = \|A\|_{1 \to \infty} m(y)
		\end{align*}
		proving $ b(x,y)\leq \|A\|_{1 \to \infty}m(x)m(y) $ for all $ x,y\in X $.
\end{proof}

\begin{proof}[Proof of Proposition~\ref{p:GF}] By decomposition into positive and negative part we can assume that $ f $ and $ g $ are positive.
	
	By the lemma above we have $ Ag \in \ell^{\infty}(X) $ and, therefore,
	\begin{align*}
	\langle f,\Delta g\rangle_{1,\infty} =	\langle f,A g\rangle_{1,\infty} -	\langle f,\Deg g\rangle_{1,\infty}.	
	\end{align*}
	Tonelli's theorem gives $ \langle f,A g\rangle_{1,\infty}=\langle A f, g\rangle_{\infty,1} $ for the positive functions $ f,g $. Moreover, since $ \Deg $ is a multiplication operator, we have  $ \langle f,\Deg g\rangle_{1,\infty}=\langle \Deg f, g\rangle_{\mathrm{abs}} $. Hence,
	\begin{align*}
	\langle f,\Delta g\rangle_{1,\infty} =	\langle Af, g\rangle_{\infty,1} -	\langle \Deg f, g\rangle_{\mathrm{abs}} =\langle \Delta f, g\rangle_{\mathrm{abs}}.	
	\end{align*}
	This finishes the proof.
\end{proof}

\section{A product rule}\label{s:product}

To prove the main theorem, Theorem~\ref{thm:main}, we need to compute the derivative of the function
\begin{align*}
s\mapsto \langle \Gamma_{\ind}(P_{t-s} f), P_{s} 1_x \rangle_{1,\infty}
\end{align*}
for fixed $ t\ge0 $ and $ x\in X $ where $1_x \in C(X)$ with $1_x(y)=1$ if $x=y$ and $1_x(y)=0$ otherwise. This will be achieved by the next proposition and the Green's formula above.

To this end, recall that $ P_{t}=e^{-tL} $, $ t\ge0 $, is a bounded positive contraction semigroup on $ \ell^{p}(X,m) $, $ p\in[1,\infty], $ which is strongly continuous for $ p\in[1,\infty) $ (where $ L $ is  the Friedrichs extension of the restriction of $ -\Delta $ to $ C_{c}(X) $). Moreover,  $ P_{t} $ maps $ \ell^{2}(X,m) $ to $D(L)\subseteq D(L^{1/2}) $, and therefore, $ \Gamma_{k}(P_{t}f)\in \ell^{1}(X,m) $ for all $ f\in \ell^{2}(X,m) $, $ k=0,1 $ and  $ t>0 $.
% Furthermore, since $ C_{c}(X)\subseteq D(L^{1/2}) $ we have also $ \Gamma_{k}(P_{t}f)\in \ell^{1}(X,m) $ for all $ f\in \ell^{2}(X,m) $,  $ k=0,1 $ and $ t\ge 0 $.

The following proposition is found in \cites{lin2015equivalent,Gong2017} under  stronger assumptions.

	\begin{proposition}\label{thm:HuaDerivative}
		Let $b$  be a locally finite graph over $(X,m)$ and $ k=0,1 $. Then, for all $ f\in C_{c}(X) $,  $0< s < t $ and $x\in X$, we have  $ \partial_{s}\Gamma_{k}(P_{t-s}f)\in \ell^{1}(X,m)  $, $ \Delta P_{s}1_{x}\in \ell^{\infty}(X) $ and
		\begin{align*}
\partial_s \langle \Gamma_{\ind}(P_{t-s}f), P_{s} 1_x \rangle_{1,\infty} = \langle\partial_s \Gamma_{\ind}(P_{t-s} f), P_{s} 1_x \rangle_{1,\infty}+ \langle \Gamma_{\ind}(P_{t-s} f), \Delta P_{s} 1_x \rangle_{1,\infty}.
		\end{align*}
	\end{proposition}
While formally the statement in the proposition above  is only the product rule, one has to ensure that one is allowed to interchange the corresponding limits. To this end, the following lemma is vital.
	
	\begin{lemma}\label{l:UniformSummableGamma} Let $b$  be a locally finite graph over $(X,m)$ and $ k=0,1 $. Then, for all $ f\in C_{c}(X) $,  $0\leq s\le t $ and $x\in X$, we have
		\begin{align*}
		\partial_s \Gamma_\ind(P_{t-s} f) = -2\Gamma_\ind(P_{t-s} f,\Delta P_{t-s} f)
		\end{align*}
		and for all $ T\ge0 $,
		\begin{align*}
 \max_{t\in [0,T]} \Gamma_\ind(P_t f),\,		  \max_{t\in [0,T]} |\partial_t \Gamma(P_t f)| \in \ell^1(X,m).
		\end{align*}
	\end{lemma}
	\begin{proof}
	To interchange the derivative $ \partial_{t} $ with $ \Gamma_\ind $ we employ the fundamental theorem of calculus and local finiteness. Then, the first statement follows by Leibniz' rule and since $ \partial_{s} P_{t-s}f=-\Delta P_{t-s}f $.
	
To show the second statement, we need a \emph{basic estimate}. 
	Since $ P_{t} =e^{-tL}$ is a contraction on $ \ell^{2}(X,m) $, we infer for $ k,l=0,1 $,
 $$   \|\Gamma_{k}(P_{t}\Delta^{l} f)\|_{1}=
 \|L^{k/2}e^{-tL}L^{l}f\|_{2}^{2}\leq \|e^{-tL}L^{k/2+l}f\|_{2}^{2} \leq\|L^{k/2+l}f\|_{2}^{2} =\|\Gm_{k}(\Delta^{l}f)\|_{1} .$$
	
 Notice that by the fundamental theorem of calculus we have for $ k=0,1 $, $ t\ge0 $ and $ f\in C_{c}(X) $
	\begin{align*}
	\Gamma_\ind(P_{t}f)-\Gamma_\ind(f)&=\int_{0}^{t}\partial_{s}\Gamma_\ind(P_{s}f)ds=2\int_{0}^{t}\Gamma_\ind(P_{s}f,\partial_{s}P_{s}f)ds\\
	&=2\int_{0}^{t}\Gamma_\ind(P_{s}f,\Delta P_{s}f)ds=2\int_{0}^{t}\Gamma_\ind(P_{s}f,P_{s}\Delta f)ds.
	\end{align*}
	
		By Cauchy-Schwarz, $2|\Gamma_\ind(f,g)| \leq 2\sqrt{\Gamma_\ind (f)} \sqrt{\Gamma_\ind (g)} \leq \Gamma_\ind (f) + \Gamma_\ind (g)$, and thus for $t \leq T$,	
		\begin{align*}
		\Gamma_\ind(P_{t}f) &\leq \Gamma_\ind(f) + \int_{0}^{T}( \Gamma_{\ind}(P_s f) + \Gamma_\ind (P_s \Delta f) ) ds .%=:h
		\end{align*}
		 Hence, taking the $  \|\cdot\|_{1} $ norm
				we get by the \emph{basic estimate} in the beginning,
		\begin{align*}
	\left	\|\max_{t\in [0,T]} \Gamma_\ind(P_{t}f)\right\|_1 &\leq \|\Gamma_\ind(f)\|_1  + \int_{0}^{T}\left(\| \Gamma_{\ind}(P_s f)\|_1  + \|\Gamma_\ind (P_s \Delta f)\|_1  \right) ds \\
		&\leq \|\Gamma_\ind(f)\|_1  + T(\|\Gamma_\ind(f)\|_1  +\|\Gamma_\ind(\Delta f))\|_1 =:C_{k}(f,T) <\infty.
		\end{align*}
		Thus, $  \max_{t\in [0,T]} \Gamma_\ind(P_t f)\in \ell_1(X,m)  $. Furthermore,
		by Cauchy-Schwarz and the equality in the statement shown above, we obtain $$ |\partial_t \Gamma_\ind (P_t f)| = 2|\Gamma_\ind(P_t f, \partial_t P_t f)| \leq 2\sqrt {\Gamma_\ind (P_t f)} \sqrt{\Gamma_\ind (P_t \Delta f)} \leq \Gamma_\ind (P_t f)+ \Gamma_\ind (P_t \Delta f). $$ 
		Hence,
		\begin{align*}
		\left\|\max_{t\in [0,T]} |\partial_t \Gamma_\ind (P_t f) | \right\|_1 \leq \left\|\max_{t\in [0,T]} \Gamma_\ind (P_t f) \right\|_1 + \left\|\max_{t\in [0,T]} \Gamma_\ind (P_t \Delta f) \right\|_1 \leq C^{(\ind)}(T,f) + C^{(\ind)}(T,\Delta f)
		\end{align*}
	implies	$  \max_{t\in [0,T]}|\partial_{t} \Gamma_\ind(P_t f)|\in \ell_1(X,m)  $.
	\end{proof}
\begin{rem}
	The proof of the lemma actually works for functions in $  D(L^{5/2})$.
\end{rem}	

\begin{proof}[Proof of Proposition~\ref{thm:HuaDerivative}]
	We recall that for a function $ f:X\times (a,b)\to\IR $ that is differentiable in the second variable and satisfies
	$f(\cdot,s)\in\ell^{1}(X,m)$  for all $ s\in (a,b) $ and $\sup_{s\in (a,b)} |\partial_{s}f(\cdot,s)|\in\ell^{1}(X,m)$, we have
	\begin{align*}
	\partial_{s}\sum_{y\in X}f(y,s)m(y)=	\sum_{y\in X}\partial_{s}f(y,s)m(y).
	\end{align*}
	We apply this to the function 
	\begin{align*}
	f_{k}(y,s)=(\Gamma_{\ind}(P_{t-s} f) P_{s}1_x) (y).
	\end{align*}

	Since $ \Gamma_{\ind}(P_{t-s} f)\in \ell^{1}(X,m) $ and $ P_{s}1_x\in \ell^{\infty}(X) $, we have $ \Gamma_{\ind}(P_{t-s} f)P_{s}1_x\in \ell^{1}(X,m) $  for all $ s\in (0,t) $. Moreover, 	we have to show that $\sup_{s\in (0,t)} |\partial_{s}f_{k}(\cdot,s)| \in \ell^{1}(X,m)$ for all $ t> 0 $. To this end we apply Leibniz rule 	
		\begin{align*}
		\partial_s f_{\ind}(y,s) = 
		   \left( (\partial_s \Gamma_{\ind}(P_{t-s} f)) P_{s}1_x
		  +\Gamma_{\ind}(P_{t-s} f) \Delta P_{s}1_x \right) (y).
		\end{align*}

As for the first part, $\sup_{s\in (0,t)}\partial_s \Gamma_{\ind}(P_{t-s} f) \in \ell^{1}(X,m)$ by the lemma above  and $0\leq  P_{s}1_{x}\leq 1 $ since $ P_{s} $ is a contraction on $ \ell^{\infty}(X) $.  

As for the second part, $\sup_{s\in (0,t)}\Gamma_{\ind}(P_{t-s} f)\in \ell^{1}(X,m)$  by the lemma above  and by local finiteness $ \Delta 1_{x}\in C_{c}(X)\subseteq \ell^{\infty}(X) $, i.e., $ |\Delta 1_{x}|\leq C  $ for some $ C\ge0 $, so, we have $|\Delta P_{t}1_{x}|=|P_t \Delta 1_{x}|\in\ell^{\infty}(X)\leq C $.

Therefore, we can exchange derivation and summation which yields
	\begin{align*}
	\partial_s \langle \Gamma_{\ind}(P_{t-s} f), P_{s} 1_x \rangle_{1,\infty}
	&= \partial_s \sum_{y\in X} f_{\ind}(y,s)m(y) = \sum_{y\in X} \partial_s f_{\ind}(y,s)m(y) \\
	&=\langle\partial_s \Gamma_{\ind}(P_{t-s} f), P_{s} 1_x \rangle_{1,\infty} 
	+\langle \Gamma_{\ind}(P_{t-s} f), \Delta P_{s} 1_x \rangle_{1,\infty}
	.
	\end{align*}
	This finishes the proof.
\end{proof}
	
	\section{Proof of the main theorem}\label{s:proof}
	
	The key ingredient to the proof is the following theorem which is a combination of the product role and Green's formula above.
	
	\begin{theorem}\label{thm:GammaDerivative}
	Let $ b $ be a locally finite graph over $ (X,m) $ that satisfies the ellipticity condition $ \mathrm{(EC)} $.	Let $\ind \in {0,1}$ and $x \in X$.
		Then, for $f \in C_c^+(X)$, 
		\[
		\partial_s \langle \Gamma_{\ind}(P_{t-s} f), P_{s} 1_x \rangle_{1,\infty} = 2\langle \Gm_{\ind+1}(P_{t-s} f), P_{s} 1_x \rangle_{\mathrm{abs}}.
		\]
		Moreover in other terms,
		\begin{align*}
		\partial_s  P_{s}\Gamma_{\ind}(P_{t-s} f)(x)=2 P_{s} \Gm_{\ind+1}(P_{t-s} f)(x).
		\end{align*}
	\end{theorem}
	\begin{proof}
		By Proposition~\ref{thm:HuaDerivative} we have
			\begin{align*}
			\partial_s \langle \Gamma_{\ind}(P_{t-s} f), P_{s} 1_x \rangle_{1,\infty} =
			 \langle\partial_s \Gamma_{\ind}(P_{t-s} f), P_{s} 1_x \rangle_{1,\infty}
			 +
			\langle \Gamma_{\ind}(P_{t-s} f), \Delta P_{s} 1_x \rangle_{1,\infty}  .
			\end{align*}
			Since $\Gm_{\ind}(P_{t-s}f), P_{s}1_{x} \in\ell^{1}(X,m) $ and $ \Delta P_{s}1_{x} = P_{s}\Delta1_{x}\in \ell^{\infty}(X)  $ (as $ \Delta1_{x}\in C_{c}(X)\subseteq D(L) $ by local finiteness), we can apply Green's formula, Proposition~\ref{p:GF}, to obtain
			\begin{align*}
			\ldots&= \langle\partial_s \Gamma_{\ind}(P_{t-s} f), P_{s} 1_x \rangle_{1,\infty} +
			 \langle\Delta \Gamma_{\ind}(P_{t-s} f),  P_{s} 1_x \rangle_{\mathrm{abs}}\\
			&=\langle  \partial_s \Gamma_{\ind}(P_{t-s} f)+\Delta \Gamma_{\ind}(P_{t-s} f),  P_{s} 1_x \rangle_{\mathrm{abs}}.
			\end{align*}
			Now, by Lemma~\ref{l:UniformSummableGamma} we have
			\begin{align*}
			\ldots
		%&= 	\langle(2 \Gm_{\ind}(P_{s}f,\partial_sP_{s}f)-\Delta \Gamma_{\ind}(P_s f)),  P_{t-s} 1_x \rangle_{\mathrm{abs}}\\
			&=
			\langle - 2 \Gm_{\ind}(P_{t-s}f,\Delta P_{t-s}f) + \Delta \Gamma_{\ind}(P_{t-s} f),  P_{s} 1_x \rangle_{\mathrm{abs}}		
			\end{align*}			
			and by definition of $ \Gm_{\ind+1} $ we arrive at
			\begin{align*}
			\ldots
			&=2\langle \Gm_{\ind+1}(P_{t-s}f),  P_{s} 1_x \rangle_{\mathrm{abs}}.
			\end{align*}
			This finishes the proof of the first statement. The ``moreover'' statement follows directly since $\langle g, P_s 1_x \rangle_{\mathrm{abs}} = m(x)P_s g(x)$ whenever the dual pairing is absolutely summable.
	\end{proof}
	
	The proof of Theorem~\ref{thm:main} follows by similar arguments as in \cite{hua2017stochastic}. We prove the following auxiliary lemma first.
	
	\begin{lemma}\label{l:aux}
			Let $ b $ be a locally finite graph over $ (X,m) $ that satisfies
			 the ellipticity condition $ \mathrm{(EC)} $
			and $CD(-K,n)$ for some $ K $ and $ n $. Then, for $f \in C_c^+(X)$ and  $ 0\leq s\leq t $,
			\begin{align*}%\label{eq:Hprime}
			\partial_{s} ( e^{2Ks}P_s \Gamma (P_{t-s} f ) ) \geq \frac 2 n P_s (\Delta P_{t-s} f)^2 e^{2Ks}  \geq \frac 2 n (P_t \Delta f)^2 e^{2Ks}.
			\end{align*}
	\end{lemma}
	\begin{proof}
		By Theorem~\ref{thm:GammaDerivative}, by $CD(-K,n)$ and since the semigroup is positive, we obtain
		\begin{align*}%\label{eq:Hprime}
		\partial_{s} ( e^{2Ks}P_s \Gamma (P_{t-s} f ) )=	2e^{2Ks}(P_s \Gamma_2 (P_{t-s} f) + K P_s \Gamma (P_{t-s} f))  \geq \frac 2 n P_s (\Delta P_{t-s} f)^2 e^{2Ks} % \geq \frac 2 n (P_t \Delta f)^2 e^{2Ks}.
		\end{align*}
		Furthermore,
		by the $ P_{s}1\leq 1 $ and the Cauchy-Schwarz inequality, we estimate $$ P_s (\Delta P_{t-s} f)^2 \geq (P_{s}(\Delta P_{t-s} f)^2)\cdot (P_{s}1)
		\ge (P_{s}((\Delta P_{t-s} f)\cdot 1))^{2}=
		(\Delta P_t f)^2 .$$
		This finishes the proof.
	\end{proof}

\begin{proof}[Proof of Theorem~\ref{thm:main}]
Let $H(s) := e^{2Ks}P_s \Gamma( P_{t-s} f)$ for  $ 0\leq s\leq t $. 

We first prove (i) $ \Longrightarrow $ (ii). 
By the lemma above we have $H'(s) \geq \frac 2 n P_s (\Delta P_{t-s} f)^2 e^{2Ks}$ and 
integrating over $s$ yields
\[
e^{2Kt}P_t \Gamma (f) - \Gamma (P_t f) =   H(t)-H(0) \geq \frac 2 n \int_0^t  P_s (\Delta P_{t-s} f)^2 e^{2Ks}ds.
\]
Rearranging the inequality proves (ii).

We next prove (ii) $ \Longrightarrow $ (iii). By the second inequality in the lemma above, we have $ (\Delta P_{t-s} f)^2 e^{2Ks}  \geq \frac 2 n (P_t \Delta f)^2 e^{2Ks} $. Therefore,
\[
\int_0^t  P_s (\Delta P_{t-s} f)^2 e^{2Ks}ds \geq    (\Delta P_t f)^2 \int_0^t e^{2Ks} ds  = \frac {e^{2Kt}-1}{2K}    (\Delta P_t f)^2.
\]
Plugging this inequality into (ii)  proves (iii).

We next prove (i) $ \Longrightarrow $ (iv),(v).
Let $G(s):=P_s(P_{t-s} f)^2$ for  $ 0\leq s\leq t $. Then, by Theorem~\ref{thm:GammaDerivative},
\[
G'(s) = 2P_s \Gamma( P_{t-s} f) = 2e^{-2Ks} H(s).
\]
By the lemma above we have $	H'(r) \geq \frac 2 n (P_t \Delta f)^2 e^{2Ks} $, $ s\leq r\leq t $, and, thus, 
\[
H(t) - \frac 2 n (\Delta P_t f)^2 \int_s^t  e^{2Kr} dr  \geq  H(s) \geq H(0) + \frac 2 n (\Delta P_t f)^2 \int_0^s e^{2Kr} dr.
\]
Thus,
\[
2e^{-2Ks} H(t)  -  \frac 2 n (\Delta P_t f)^2 \frac{e^{2K(t-s)}-1}{K}   \geq  G'(s) \geq 2e^{-2Ks} H(0) +  \frac{2}{n} (\Delta P_t f)^2 \frac{1 - e^{-2Ks}}K.
\]
Integrating with respect to  $ s $ from $ 0 $ to $t$ yields
\begin{align*}
 \frac {1 - e^{-2Kt}}K H(t) - \frac{ e^{2Kt}-1 -2K}{nK^2}  (\Delta P_t f)^2   &\geq G(t)-G(0) \\&\geq  \frac {1 - e^{-2Kt}}K H(0) +    \frac{e^{-2Kt} - 1 + 2Kt}{nK^2}  (\Delta P_t f)^2.
\end{align*}
Plugging in $G(t)$, $G(0)$, $H(t)$ and $H(0)$ proves (iv) and (v) respectively.

The implication (iii) $ \Longrightarrow  $ (i) follows by taking the derivative with respect to $ t $ at $t=0$. Specifically, one subtracts $ \Gamma(f) $ on each side of the inequality, divides by $ t $ and lets $ t\to 0 $.

The implications (iv) $ \Longrightarrow  $  (i) and (v) $ \Longrightarrow  $  (i) follow by taking the derivative with respect to $ t $ twice at $t=0$. Specifically, consider the Taylor expansions, cf. \cite{KLMST16}, of the term on the left hand side and the first term on the right hand side at $ t=0 $, 
\begin{align*}
  P_{t}f^{2}-(P_{t}f)^{2} &=2t\Gamma(f)+t^{2}(\Delta\Gamma(f)+2\Gamma(f,\Delta f))	+o(t^{2}),\\
   \frac{e^{2Kt}-1}{K}P_{t}\Gamma(f)&= 2t\Gamma(f)+2t^{2}\left(\Delta\Gamma(f)+K\Gamma(f) \right)+o(t^{2}).
\end{align*}

Subtracting $2 t\Gamma(f)$ on both sides, dividing by $ t^{2} $ and taking the limit $ t\to 0 $ yields the implication (iv) $ \Longrightarrow  $  (i). To see the implication (v) $ \Longrightarrow  $  (i), consider the Taylor expansion of the first term of the right hand side of (v) instead, i.e.,
 \begin{align*}
  \frac{1-e^{-2Kt}}{K}\Gamma(P_{t}f)&= 2t\Gamma(f)+2t^{2}\left(2\Gamma(f,\Delta f)-K\Gamma(f) \right)+o(t^{2}).
 \end{align*}

 Now, again subtracting  $2 t\Gamma(f) $, dividing by $ t^{2} $ and taking the limit $ t\to 0 $ yields the implication (v) $ \Longrightarrow  $  (i).
\end{proof}

%	\subsection{Semigroup characterization}

	\section{Completeness, stochastic completeness and finite measure}\label{s:appl}
	In this section, we present two applications of the characterization above. This concerns the relationship between completeness and stochastic completeness under a curvature dimension inequality.
	
	To this end, recall that a graph is said to have the \emph{Feller property} if $$  P_{t} C_{0}(X) \subseteq  C_{0}(X),  $$ where $ C_{0}(X) $ is the closure of $ C_{c}(X) $ with respect to $ \|\cdot\|_{\infty} $.
	Furthermore, a graph is
	called \emph{stochastically complete} if
	\begin{align*}
	P_{t}1=1,
	\end{align*}
	where $ 1 $ denotes the constant function $ 1 .$ 
	Finally, we call a graph \emph{complete}
	if there is an increasing nonnegative sequence of functions $\eta_k \in C_c(V)$, $ k\in\N $, that converge pointwise $\eta_k \to 1$ and
	$$
	\Gamma(\eta_k) \to 0 ,\qquad k\to\infty.
	$$

In \cite{hua2017stochastic} it is shown that under a lower curvature bound, completeness implies stochastic completeness. Next, we  prove that the converse is true as well in case of non-negative curvature when the graph also satisfies the ellipticity condition $ \mathrm{(EC)} $ and  the Feller property.
	
	\begin{theorem} 	Let $ b $ be a locally finite graph over $ (X,m) $ that satisfies the ellipticity condition $ \mathrm{(EC)} $ and $CD(0,\infty)$.  If the graph  is stochastically complete and satisfies the Feller property, then  the graph 	is complete.
	\end{theorem}

	\begin{proof}
	To prove completeness, we show that for every finite $S\subseteq X$ and $\eps >0$ there is  $\eta =\eta_\eps^S \in C_c(V)$ such that $\eta=1$ on $S$ and $\Gamma (\eta) \leq \eps$.
		Let  $t=2/\eps$. Due to stochastic completeness, we can choose $U \subseteq X$ finite such that $P_t 1_U  \geq 3/4$ on $S$. 
		Due to the  Feller property we can choose  $W \subseteq X$ finite such that $P_t 1_U  \leq 1/4$ for $X\setminus W$.
		Let
		\begin{align*}
		\eta := 1 \wedge (2P_t 1_U - 1/2)_+
		\end{align*}
		where $u_+ = u \vee 0$.
		Then, $\supp \eta \subset W$ and thus, $\eta \in C_c(X)$ and $\eta = 1$ on $S$. 	By Theorem~\ref{thm:main}~(v) the assumption  $CD(0,\infty)$ implies
		\begin{align*}
		2t\Gamma (P_t f) \leq P_t f^2 - (P_t f)^2
		\end{align*}
		for all compactly supported $f$. Thus, we estimate using the definition of $ \eta $, the inequality just mentioned and $ P_{t}1_{U}^2\leq 1 $
		$$\Gamma (\eta) \leq \Gamma(2P_t 1_U) = 4 \Gamma(P_t 1_U)\leq \frac{2}{t}(P_{t}1_{U}^{2}-(P_{t}1_{U})^{2}) \leq \frac 2 {t} = \eps.$$
		This finishes the proof.
	\end{proof}

		\begin{theorem} 	Let $ b $ be a locally finite, connected graph over $ (X,m) $ that satisfies the ellipticity condition $ \mathrm{(EC)} $ and  $CD(K,\infty)$ for some $K>0$. Furthermore, assume that the graphs is complete, then the graphs has finite measure.
	\end{theorem}
We want to point out that the same result with slightly different assumptions has been independently established  in the context of proving Buser's inequality \cite{liu2018Buser}.	
	
	\begin{proof}
		Let $x \in X$ and $ \eps>0 $.
		Due to completeness, there exists $\eta \in C_c(X)$ such that $\eta(x)=1$ and $\Gamma (\eta) < \eps$.
		Observe that by Jensen's inequality we have $$ [\Delta g (x)]^2\leq\left(\frac 1{m(x)}\sum_{y\sim x}b(x,y)|g(x)-g(y)|\right)^{2}\leq  2 \Deg(x) \Gamma (g) (x) .$$
				Therefore, by Theorem~\ref{thm:main}~(ii),
		\begin{align*}
		(\partial_t P_t \eta)^2(x) \leq 2 \Deg(x) \Gamma( P_t \eta)(x) \leq  2 \Deg(x) e^{-2Kt} P_t (\Gamma \eta) (x) \leq 2 \Deg(x) e^{-2Kt} \eps.
		\end{align*}
		Taking square root and integrating  over time yields by the fundamental theorem of calculus
		\begin{align*}\label{eq:etaIntegrate}
		\eta(x) - \lim_{t\to\infty}P_t \eta(x) \leq \frac{\sqrt{2\eps\Deg(x)}}K.
		\end{align*}
			Now, suppose $m(X) = \infty$.
		Due to connectedness, by  the spectral theorem we obtain
		$$ \lim_{t\to\infty}P_t \eta(x) =0.$$ 
		Since additionally $\eta(x)=1$, the two (in)equalities above yield
		\[
		1 \leq \frac{\sqrt{2\eps\Deg(x)}}K.
		\]
		This is a contradiction for $\eps$  small enough. Hence, the assumption  $m(X) = \infty$ is wrong which finishes the proof.
	\end{proof}
	
	Combining the above theorems yields the following corollary for \emph{combinatorial graphs}, i.e., graphs where $ b $ takes the values $ 0 $ or $ 1 $ and $ m\equiv 1 $.
	\begin{corollary}
		Let $ b $ be a connected combinatorial graph and $ m\equiv 1 $ satisfying $CD(K,\infty)$ for some $K>0$ which  is stochastically complete.
		Then the graph is finite.
	\end{corollary}

	Surprisingly, it is still unclear if the stochastic completeness assumption is necessary.

\section{Examples}

In this section, we provide examples for graphs satisfying gradient estimates by Theorem~\ref{thm:main}, which cannot be proven with the methods from \cite{lin2015equivalent} or \cite{Gong2017} or \cite{hua2017stochastic}.
Particularly, we will construct graphs with a lower curvature bound satisfying (EC), but having unbounded vertex degree and a degenerate vertex measure, meaning $\inf_x m(x) = 0$.
We will do this by glueing together a graph with bounded vertex degree and degenerate measure, with a graph with unbounded vertex degree and non-degenerate vertex measure.

We first construct a graph with degenerate measure, bounded vertex degree, and  satisfying (EC).
\begin{example}\label{ex:degenerateMeasure}
Let $G=(\N,b,m)$ with $b(n-1,n)=b(n,n-1) = 1/n$, and $b(n,m)=0$ if $|m-n|>1$, and
\[
m(n) = \begin{cases}
1&:n \mbox{ even,}\\
1/n&: n \mbox{ odd.}
\end{cases}
\]
It is easy to check that the vertex degree is bounded, and that (EC) is satisfied, and that the vertex measure is degenerate. Due to bounded vertex degree, the graph has to satisfy $CD(K,\infty)$ for some $K \in \R$ by \cite[Theorem~1.3]{lin2010ricci}.
\end{example}
	
We next construct an example with non-negative curvature and unbounded vertex degree.	
	
\begin{example} \label{ex:unbounded}
It is shown in
\cite{cushing2018curvature} that the following antitree satisfies $CD(0,\infty)$. The antitree $G=(V,b,m)$ is given by $V=\bigcup V_k$ with $|V_k|=k$ and $k \in \N$ and $V_i \cap V_j = \emptyset$ whenever $i \neq j$.
Two vertices $x \in V_i$ and $y \in V_j$ are adjacent iff $|i-j| \leq 1$. In this case, we set $b(x,y)=1$, otherwise $b(x,y)=0$.
The vertex measure $m$ is set to be constant one.
It is clear that (EC) is satisfied and that the vertex degree is unbounded.
\end{example}	
	
We finally construct examples by glueing togehter the graphs explained above.

\begin{example}
Let $G_1 = (V_1,b_1,m_1)$ be the graph from Example~\ref{ex:degenerateMeasure} and $G_2=(V_2,b_2,m_2)$ be the graph from Example~\ref{ex:unbounded}. We suppose that  $V_1 \cap V_2 = \emptyset$.
We construct $G=(V_1 \cup V_2, b,m)$ with $m(v) = m_i(v)$ if $v \in V_i$, and
$b(x,y)=b_i(x,y)$ if $x,y \in V_i$ for some $i \in \{1,2\}$, and $b(x,y) = 0$ if $x \in V_1$ and $y \in V_2$ except for finitely many exceptions where $b(x,y)$ can be positive.

Then, $G$ has a lower curvature bound, satisfies (EC), has unbounded vertex degree and degenerate measure.
\end{example}

\section*{Acknowledgments}
M.K. wants to thank the DFG for financial support. F.M. wants to thank the MPI MiS Leipzig for financial support. Both authors want to thank Melchior Wirth for valuable discussions.

\bibliographystyle{alpha}	
\bibliography{Bibliography}{}

	Matthias Keller,\\
	Department of Mathematics,
	University of Potsdam, Potsdam, Germany\\
	\texttt{matthias.keller@uni-potsdam.de}\\

	Florentin Münch, \\
	MPI MiS Leipzig, Leipzig, Germany\\
	\texttt{florentin.muench@mis.mpg.de}\\
	
	\end{document}